\documentclass[11pt]{amsart}
\usepackage[left=2cm,top=2cm,right=3cm,nofoot]{geometry}
\usepackage{amsmath,mathtools}%
\usepackage{amsfonts}%
\usepackage{amssymb}%
\usepackage{graphicx}
\usepackage{esint}
\usepackage{hyperref}
\usepackage{enumitem}
\usepackage{accents}
\usepackage{etoolbox}
\patchcmd{\subsection}{-.5em}{.5em}{}{}
\newtheorem{theorem}{Theorem}[section]
\theoremstyle{plain}

\newtheorem{lemma}[theorem]{Lemma}

\newtheorem{remark}[theorem]{Remark}

\numberwithin{equation}{section}
\theoremstyle{plain}

\usepackage{etoolbox}
\AtEndEnvironment{proof}{\setcounter{claim}{0}}

\newcommand{\Sp}{\mathbb{S}}
\newcommand{\re}{\mathbb{R}}

\begin{document}
\title[Local bifurcation diagrams and degenerate solutions ]{Local bifurcation diagrams and degenerate solutions of Yamabe-type equations}

\author{Sahid Bernab\'{e} Catal\'{a}n}
\address{Centro de Investigaci\'{o}n en Matem\'{a}ticas, CIMAT, Calle Jalisco s/n, 36023 Guanajuato, Guanajuato, M\'{e}xico}
\email{sahid.bernabe@cimat.mx}

\author{Jimmy Petean}
\address{Centro de Investigaci\'{o}n en Matem\'{a}ticas, CIMAT, Calle Jalisco s/n, 36023 Guanajuato, Guanajuato, M\'{e}xico}
\email{jimmy@cimat.mx}

\begin{abstract} We  study positive solutions of the equation  $-\Delta_g u + \lambda u = \lambda u^q$, with $\lambda >0$, $q>1$ on the round sphere $\Sp^n$. We reduce the equation to an ordinary differential equation by considering 
isoparametric functions and apply bifurcation theory. We study when the corresponding bifurcation points are transcritical. We apply this result to show the existence of degenerate solutions to the equation and to study multiplicity results for conformal constant scalar curvature metrics.

\end{abstract}
\maketitle

\section{Introduction}

On a closed Riemannian manifold $(M^n ,g)$ of dimension $n\geq 3$ we
consider the following {\it Yamabe type} equation
\begin{equation}\label{equation}
-\Delta_g u+ \lambda u= \lambda u^{q},
\end{equation}

\noindent
with $\lambda >0$ and $q>1$. Let $p_n=\frac{n+2}{n-2}$. Equation (\ref{equation}) is said to be critical if $q=p_n$, 
subcritical if $q<p_n $ and supercritical if $q>p_n$. The equation is very well-known and has been extensively studied 
in the last decades.  The critical case appears in Riemannian geometry when trying to solve the problem of
finding metrics of constant scalar curvature in a given conformal class of metrics; what is known as the Yamabe problem. 
Namely, if a metric $h$ conformal to $g$ is expressed as $h=u^{p_n -1} g$ for a positive function $u$ on $M$, then the 
scalar curvature of $h$, $s_h$, is equal to a constant $\lambda \in {\bf R}$ if and only if $u$ solves 
the Yamabe equation:

\begin{equation}\label{YamabeEquation}
- a_n \Delta_g u+ s_g u= \lambda u^{p_n},
\end{equation}

\noindent
where $a_n = \frac{4(n-1)}{n-2}$. A fundamental result obtained in several steps by H. Yamabe \cite{Yamabe}, N. Trudinger
\cite{Trudinger}, T. Aubin \cite{Aubin} and R. Schoen \cite{Schoen0} states  that there is always at least one positive solution
of the Yamabe equation, for any closed Riemannian manifold. Since the Yamabe equation is conformally invariant,  
in order to understand the space of
solutions we can therefore assume that $s_g$ is constant.  In the non-positive case Equation (\ref{YamabeEquation})
only has the constant solution and therefore we assume that $s_g$ is positive. Then Equation (\ref{YamabeEquation}) is of the form 
(\ref{equation}) with $q=p_n$ and $\lambda =\frac{s_g}{a_n}$.

\vspace{.5cm}

We will study Equation  (\ref{equation}) on the  sphere of dimension $n$ with the metric of constant sectional curvature 1, $(\Sp^n  , g_0 )$. If we consider a Riemannian product 
with the sphere, $(M \times \Sp^n  , g+g_0 )$,  we could  search
for solutions of the Yamabe equation which are functions that depend only on the sphere. Similarly if a Riemannian manifold
is the total space of a harmonic Riemannian submersion over the sphere one can consider solutions which are constant along the 
fibers. Then the Yamabe equation reduces to a subcritical equation on the sphere (see for instance \cite{Bettiol, deLima, Otoba}). Therefore
the results we will prove on positive solutions to Equation (\ref{YamabeEquation}) also give multiplicity results for conformal constant
scalar curvature metrics on the total spaces of such fibrations.

Many multiplicity results for the Yamabe type equation (\ref{equation}) have been obtained using bifurcation techniques.
In bifurcation theory one considers a family of solutions (which one considers as the family of {\it trivial}
solutions) and tries to understand the connected component of this family of trivial solutions (in the space of
all solutions). In particular one wants to understand which are the {\it bifurcation points} of the family; the
elements of the family which are accumulation points of solutions which are not in the family. 
By  {\it Local Bifurcation} we mean to understand which are the bifurcation points in the family 
and what is the space of solutions in a neighborhood of a bifurcation point. Local bifurcation for Equation (\ref{equation})
 was studied for instance in \cite{Bettiol, BettiolPiccione, deLima, Otoba, Petean}.  We mean by {\it Global Bifurcation} the study of the 
whole connected component, which is a more difficult problem.  Global bifurcation results for Equation
(\ref{equation})  have been obtained in \cite{AJJ, BettiolPiccione2, Henry}. All these results on global 
bifurcation have been obtained using a reduction of the equation to an ordinary dfferential equation, as we will do in this article.

\vspace{.8cm}

Recall that a   function $f:M\rightarrow[t_0,t_1]$, on a closed Riemannian manifold $(M,g)$,  is called {\it isoparametric} if  $|\nabla f|^2=b(f)$, $\Delta f=a(f)$,  for 
some smooth functions $a,b: [t_0 , t_1 ] \rightarrow \re$. Isoparametric functions on general Riemannian manifolds were considered by Q-M Wang in
\cite{Wang}, following the classical work by E. Cartan \cite{Cartan}, B. Segre \cite{Segre}, T. Levi-Civita \cite{Civita}  in the case of space forms.  An isoparametric function is called {\it proper} if its regular
level sets are connected.  For instance if by $x_1 : (\Sp^n  ,g_0 ) \rightarrow \re$ we mean the projection in the first 
coordinate, then $x_1$ is a proper isoparametric function while $x_1^2$ is an isoparametric fuction which is not proper.
More details on isoparametric functions can be found in \cite{Tang, Ge, m1, m2}.

Assume that there is an isoparametric function $f:M \rightarrow \re$ and look for solutions of
Equation (\ref{equation}) of the form $v=\varphi \circ f$, where $\varphi :\re \rightarrow \re_{>0}$. 
We will call functions $v$ of this form {\it $f$-invariant}. 
We will denote by $C_f^{2,\alpha}$ the space of $C^{2,\alpha}$ functions on $M$ which are $f$-invariant. 

We will consider  (\ref{equation}) as an equation on $(u,\lambda )$
and study
solutions bifurcating from the family of trivial solutions $\lambda \mapsto (1, \lambda)$. In this context one says that the family
of solutions is {\it locally rigid} at $(1,\lambda_0 )$ if there is a neighborhood $U$ of $(1,\lambda_0 )$ such that if $(u, \lambda ) \in U$
is a solution then $u=1$. In the other case, namely if there is a sequence of non-constant solutions approaching $(1,\lambda_0 )$, we
say that $(1,\lambda_0 )$ is a {\it bifurcation point} for the family.

Using the well-known theory of local 
bifurcation for simple eigenvalues, which started  with the classical work of M. G. Crandall, P. H. Rabinowitz in \cite{Crandall}, 
one can prove the following result (it is proved explicitly in this case in \cite{AJJ}):

\begin{theorem}\label{LocalBifurcationTheorem} Let $f$ be a proper isoparametric function on $(M,g)$. Let $\mu_k$, $k\geq 1$, be the (negative) eigenvalues of
$\Delta_g |_{C_f^{2,\alpha}}$. For any $q>1$ let $\lambda_k = \lambda_k (q) = \frac{- \mu_k}{q-1}$. For any $k\geq 1$ there is a continuous 
  branch
$t \mapsto (u(t), \lambda (t) )$, $t\in (-\varepsilon , \varepsilon )$, of $f$-invariant solutions of Equation (\ref{equation}) so that
$\lambda (0) =\lambda_k$, $u(0)=1$ and $u(t) \neq 1$ if $t\neq 0$.
\end{theorem}

We will  consider the local behavior of the non-trivial branch of solutions appearing at the bifurcation points. Recall that
the bifurcation point $(1, \lambda_k )$ is called {\it transcritical} if the nontrivial branch $s(t)=(u(t), \lambda (t))$
verifies $\lambda ' (0) \neq 0$. In this case the branches $s(t)$, $t>0$ and $s(t)$, $t<0$, are on different sides of
the vertical line $\lambda (t) =\lambda_k$. If $\lambda '(0) = 0$ and $\lambda '' (0) \neq 0$ then both branches
stay on the same region $\lambda < \lambda_k$ or $\lambda > \lambda_k$. 

\vspace{1cm}

From now on we will focus on the case of the round sphere $(\Sp^n  ,g_0 )$. 
The simplest isoparametric functions on the sphere are obtained by
considering  the isometric cohomogeneity one  $O(n)$-action  fixing an axis. The $O(n)$-invariant functions will be called radial (with 
respect to the fixed axis).  A linear 
function on $\re^{n+1}$ invariant by the action (and restricted to the sphere)  gives a proper  isoparametric function on the sphere. It is  actually an eigenfunction for
the Laplacian, corresponding to the first nontrivial eigenvalue. The first result about  the local behavior of the non-trivial branches of solutions 
in this situation was obtained by H. Brezis and
Y.Y. Li in \cite{Brezis}. They considered the case of the first radial eigenfunction, $\lambda_1$.
It was shown  in \cite[Theorem 3, Remarks 7 and 8]{Brezis} that if $(u(s), \lambda (s))$ is the
branch of non-trivial solutions bifurcating at $(1,\lambda_1 )$, then   $\lambda '(0) =0$. Moreover  $\lambda ''(0) \neq 0$,  and
the sign of $\lambda ''(0)$ depends on whether the equation is sub or supercritical.

General isoparametric functions on the sphere are the restrictions   of Cartan-M\"{u}nzner polynomials on $\re^{n+1}$ (\cite{m1,m2}). 
For  a polynomial  $F$ in $\re^{n+1}$ we let $d$ be the degree of $F$. Then $F$ is called a  Cartan-M\"{u}nzner polynomial if it
 satisfies the  Cartan-M\"{u}nzner equations:

$$ \langle \nabla F , \nabla F \rangle = d^2 \| x \|^{2d -2} $$

$$ \Delta F = \frac{1}{2} c d^2 \| x \|^{d -2} ,$$

\noindent
where $c$ is a constant, which we will describe below.
Then $f=F_{|_{S^n}}$ is an isoparametric function
on the sphere: it verifies

$$ \langle \nabla f , \nabla f \rangle = d^2 (1-f^2 ) $$

$$ \Delta f = -d(n+d-1)f +\frac{1}{2} c d^2 .$$

The constant $c$ has a geometric meaning (see \cite{m1,m2}): $f: S^n \rightarrow [-1,1]$ and its only critical values are -1 and 1. For $t \in (-1,1)$ $f^{-1}(t)$ is called an {\it isoparametric hypersurface} of
degree $d$. $d$ can only  take the values 1, 2, 3, 4 or 6. $f^{-1}(t)$ has constant principal curvatures, and the number of distinct principal curvatures is $d$. In case $d$
is odd all the distinct principal curvatures have the same multiplicity. If $d=2,4$ or 6 then half of principal curvatures have multiplicity $m_1$  and the other half have
multiplicity $m_2$. The constant $c$ is $m_2 - m_1$. Note that $-F$ also satisfies the  Cartan-M\"{u}nzner equations replacing $c$ by $-c$, which amounts to exchanging $m_1$ 
and $m_2$. Then we will always assume that $c \leq 0$.

We will denote an isoparametric function on the sphere by $f^{d,c}$ (even when there are different  Cartan-M\"{u}nzner polynomials with the
same values of $d$ and $c$).

Note that if $\varphi \in C^2 [-1,1]$ then $\Delta ( \varphi \circ f^{d,c} ) = d^2 (1 - (f^{d,c})^2 ) \  (  \varphi '' \circ f^{d,c}  ) \ + (-d(n +d -1)f^{d,c} +(1/2)cd^2 )  \   ( \varphi ' \circ f^{d,c} )$. And therefore
$\varphi \circ f^{d,c}$ satisfies Equation (\ref{equation}) if and only if $\varphi$ satisfies 

\begin{equation}\label{Gegenbauer}
d^2 (1-t^2 ) \varphi '' (t) + (-d(n+d-1)t + (1/2)cd^2) \varphi '(t) - \lambda ( \varphi (t) - \varphi (t) ^q ) =0.
\end{equation}

Recall that the eigenvalues of the Laplace operator on the sphere are $\mu_i =- i(n+i-1)$.  
The eigenvalues of the Laplace operator restricted to the space of $f^{d,c}$-invariant functions are $\mu_{di} =- di(n+di-1)$ (see \cite[Lemma 3.4]{Henry}). The
space of $f^{d,c}$-invariant eigenfunctions corresponding to the eigenvalue $\mu_{di}$ has dimension 1: let
$\varphi_i (d,c) \neq 0$ be  such an eigenfunction and write $\varphi_i (d,c) = \varphi  \circ f^{d,c}$. Then $\varphi$ satisfies

\begin{equation}\label{Linear}
d^2 (1-t^2 ) \varphi '' (t) + (-d(n+d-1)t + (1/2)cd^2) \varphi '(t) +di (n+di -1)  \varphi (t) =0.
\end{equation}

If we pick $\alpha , \beta$ by solving 

$$\beta - \alpha = (1/2)c $$

$$\alpha + \beta +2 = \frac{n+d-1}{d},$$

Then Equation (\ref{Linear}) becomes

\begin{equation}\label{LinearJacobi}
 (1-t^2 ) \varphi '' (t) + (-(\alpha +\beta +2)t + \beta -\alpha ) \varphi '(t) +i (i + \alpha +\beta +1)  \varphi (t) =0,
\end{equation}

\noindent
which is the usual form of the classical Jacobi equation. Note that since we are assuming $c\leq 0$ we have that $\alpha \geq \beta$ and we also 
have that $\alpha + \beta +1 > 0$. 

Let
$(u^{d,c}_i (s), \lambda^{d,c}_i  (s))$ be the
branch of non-trivial $f^{d,c}$-invariant  solutions bifurcating at $(1, \frac{-\mu_{di}}{q-1} )$, as in Theorem \ref{LocalBifurcationTheorem}.
It is shown in  \cite[Section 3]{Petean} that 

$$\frac{d \lambda^{d,c}_i }{ds}   (0) = C \int_{\Sp^n} \varphi^3 ,$$

\noindent
where $C \neq 0$ is a constant and $\varphi \neq 0$ is a $f^{k,c}$-invariant eigenfunction of the Laplacian, $-\Delta \varphi = ki(n+ki-1) \varphi $.
Using this, it is shown in \cite{Petean} that
for the first associated eigenfunction one has $\frac{d \lambda_1^{d,c}}{ds}  (0)=0$ if and only if $c=0$. And in case $c=0$, one has that, for the second 
associated eigenfunction, $\frac{d \lambda^{d,0}_2  }{ds} (0) \neq 0$.

We will prove:

\begin{theorem}\label{Main} If $c=0$ and  $i \geq 1$ the branch of nontrivial bifurcating solutions  $(u^{d,0}_{i} (s), \lambda^{d,0}_{i}  (s))$ 
of Equation (\ref{equation}) appearing at $(1, \frac{-\mu_{di}}{q-1} )$ verifies 
$\frac{d \lambda^{d,0}_{i}}{ds} (0) =0$ if and only if $i$ is odd. 

If $c\neq 0$ then for any $i$ the branch of nontrivial solutions $(u^{d,c}_{i} (s), \lambda^{d,c}_{i}  (s))$ appearing at $(1,\frac{-\mu_{di}}{q-1}) $ verifies 
$\frac{d \lambda^{d,c}_{i}}{ds}  (0) \neq 0$.
\end{theorem}

The proof of  Theorem \ref{Main} is very different from the one used  in \cite{Petean} for the cases mentioned above. Note that 

$$\varphi_1 (d,c) = f^{d,c} -\frac{cd^2 }{2(n+d-1)},$$

\noindent
is a $\mu_{d}$-eigenvalue on the sphere, and also an isoparametric function.
 Using this to compute $\Delta \varphi_1 (d,c)^2$ and integration by parts it is shown in 
\cite{Petean} that 

$$(\mu_d +2d^2 )\int_{\Sp^n} \varphi_1 (d,c)^3 =-\frac{2cd^3}{n+d-1} \int_{\Sp^n} \varphi_1 (d,c)^2 ,$$

\noindent
which implies the first above mentioned result.
When $c=0$ then one can see that $\varphi_2 (d,0)$ is also an isoparametric function (although it is not proper). For instance in the case of the $O(n)$-invariant
isoparametric functions discussed above, we have that  $\varphi_1 (d,0)$ is a linear function, $x$, and $\varphi_2 (d,0)$ is $x^2 -1/(n+1)$. Then in a similar
way it is shown in \cite{Petean} that 

$$\int_{\Sp^n} \varphi_2 (d,0)^3 \neq 0.$$

It is easy to see that all the other  $f^{k,c}$-invariant eigenfunctions are not isoparametric, and therefore the previous argument does not work. 
Instead, in the article we will work directly with Equation (\ref{Gegenbauer}), and we will prove Theorem \ref{Main} by computing certain integrals related
to the Jacobi  polynomials. This will be carried out in Section 3. In Section 2 we will have a  discussion on
the Jacobi equations and polynomials, and their linearization formulas, that are necessary to compute the required integrals.

In Section 4 we will discuss how to apply Theorem \ref{Main} to prove existence of degenerate solutions of Equation (\ref{equation}). Recall that a solution
is called {\it degenerate} if the linearized equation has a  nontrivial kernel. For instance, in the context of Theorem \ref{LocalBifurcationTheorem} a  constant solution $(1, \lambda )$
 is degenerate if and only if 
$(1,\lambda )$ is a bifurcation point. Degenerate solutions are important to understand the global picture of the family of solution which appear
through bifurcation, since the turning points of the curves of nontrivial solutions are degenerate solutions. 
We will need to impose conditions on $q$. 
For any isoparametric function $ f^{d,c}$ the set of critical points has two connected components  $M_1 =  (f^{d,c})^{-1} (-1) $ and  $ M_2 = f(^{d,c})^{-1} (1)$. 
Let $m_i$ is the dimension of  $M_i$ and let $m=\min \{m_1 , m_2 \} \leq n-2$. Then we let
$q_f = \frac{n-m +2}{n-m-2}$, $q_f = \infty$ in case $m=n-2$. Note that if $m>0$ then $q_f > p_n$. For the
next result we will ask that $q< q_f$. If $m>0$ the result applies then to supercritical equations.

We will prove:

\begin{theorem}\label{DegenerateSolutions} Let $f^{d,c}$ be an isoparametric function on the sphere as above. If  $c=0$, then for any positive even integer $k$
there exists a degenerate $f^{d,0}$-invariant positive solution of Equation (\ref{equation}) (for some $\lambda >0$) for which the set of critical points  has exactly $k+1$ connected components. 

If $c\neq 0$ then for any positive  integer $k$
there exists a degenerate $f^{d,c}$-invariant solution of Equation (\ref{equation}) (for some $\lambda >0$)  for  which the set of critical points has exactly $k+1$ connected components.

\end{theorem}

\begin{remark} A $f^{d,c}$-invariant positive solution, $\varphi \circ  f^{d,c}$, of Equation (\ref{equation}) is given by a positve solution, $\varphi$,  of Equation (\ref{Gegenbauer}). The critical points
of $\varphi \circ  f^{d,c}$ are the critical points of $f^{d,c}$, which has the two connected components $M_1$ and $M_2$ mentioned above, and the preimage by $f^{d,c}$ of the 
critical points of $\varphi$. Therefore Theorem \ref{DegenerateSolutions} is equivalent to prove the existence of positive solutions of Equation (\ref{Gegenbauer}) with $k-1$ critical
points. Note that for non-constant positive solutions $\varphi$ of Equation (\ref{Gegenbauer}), 1 and -1 are not critical ponts.

\end{remark}

\begin{remark} Once we fixed a solution of the linearized Equation (\ref{Linear}) we will actually compute the sign of $\frac{d \lambda^{d,c}_{i}}{ds}  (0)$. In
this way we will actually be able to say if the corresponding degenerate solution from Theorem \ref{DegenerateSolutions}  has a local maximum or a local minimum at 
$M_1$ and $M_2$.

\end{remark}

\section{Jacobi polynomials and their linearization coefficients}

In this section we will review results about Jacobi polynomials.  This will allow us to give the computations of
certain integrals involving the Jacobi polynomials, which are necessary to understand the local behaviour of the bifurcation branches discussed in the introduction
and in the next section. Most of the results we will mention appear in the book by R. Askey \cite{Askey}, and we will
follow the same notation and normalizations used in the book.

The Jacobi differential equation is

\begin{equation}\label{JacobiEquation}
(1-t^2 )y'' (t) +(\beta - \alpha -(\alpha +\beta +2)t) y'(t) + k(k+\alpha + \beta +1) y(t) =0.
\end{equation}

\noindent
In the equation $k$ is a positive integer, and $\alpha , \beta >-1$ are real numbers. The equation is defined in the interval
$[-1,1]$. Solutions are $C^2$-functions which verify

\begin{equation}\label{-1}
(\beta - \alpha +(\alpha +\beta +2)) y'(-1) + k(k+\alpha + \beta +1) y(-1) =0
\end{equation}

and

\begin{equation}\label{1}
 (\beta - \alpha -(\alpha +\beta +2)) y'(1) + k(k+\alpha + \beta +1) y(1) =0.
\end{equation}

It is of course an eigenvalue problem. For $\lambda \in \re$ consider the equation

\begin{equation}\label{JacobiEquationL}
(1-t^2 )y'' (t) +(\beta - \alpha -(\alpha +\beta +2)t) y'(t) + \lambda  y(t) =0.
\end{equation}

\noindent
The space of solutions of Equation (\ref{JacobiEquationL}) satisfying the initial condition (\ref{-1}) has dimension 1. One looks for the values of $\lambda$ for which the 
corresponding solutions are defined in the whole interval [-1,1], and  satisfy (\ref{1}). Constant functions satisfy  Equation (\ref{JacobiEquationL}) for $\lambda =0$.
It is also easy to check that (multiples of)  $y(t) = t - \frac{\beta -\alpha}{\alpha +\beta +2}$ satisfies
Equation (\ref{JacobiEquationL}) for $\lambda =1(1+\alpha +\beta +1)$. 

Let $L(y)= (1-t^2 )y'' (t) +(\beta - \alpha -(\alpha +\beta +2)t) y'(t)$.  Then we point
out that for $k\geq 2$
we have that 

$$L(t^k ) =  -k(k + \alpha +\beta  +1) t^k  +k (\beta -\alpha )t^{k-1} +k(k-1)t^{k-2}.$$ 

Note that $L(t^k ) + k(k + \alpha +\beta  +1) t^k $ is a polynomial of degree 
$k-1$ and if $k\neq j$ then $L(t^k ) + j(j + \alpha +\beta  +1) t^k$ is a polynomial of degree $k$. 

It then follows easily by induction that for each $k$ there is a polynomial of degree $k$ satisfying Equation (\ref{JacobiEquation}).
Such polynomial is called
a Jacobi polynomial, denoted as $P_k^{(\alpha , \beta )}$.

\begin{remark}\label{OddEven}
Note that with the same argument one can see that if $\alpha = \beta$ then if
$k$ is odd then $P_k^{(\alpha , \beta )}$ is an odd polynomial and if $k$ is even then $P_k^{(\alpha , \beta )}$ is an even polynomial. This is
not the case if $\alpha \neq \beta$, as we have seen for instance for $P_1^{(\alpha , \beta )}$.
\end{remark}

For $\alpha , \beta$ fixed, the Jacobi polynomials $P_k^{(\alpha , \beta )}$, $k\geq 1$, are orthogonal with respect to the weight
$(1-t)^{\alpha} (1+t)^{\beta}$: i.e. if $j \neq k$ then

$$\int_{-1}^{1} P_k^{(\alpha , \beta )}P_j^{(\alpha , \beta )} (1-t)^{\alpha} (1+t)^{\beta} dt =0.$$

One can check this easily. First note that for any $u,v \in C^2 [-1,1]$ integrating by parts one obtains that:

\begin{equation}\label{Parts}
\int_{-1}^{1}  ( (1-t)^{\alpha +1} (1+t)^{\beta +1}  u'  )'  \ \ v  \ \ dt = \int_{-1}^{1}  ( (1-t)^{\alpha +1} (1+t)^{\beta +1}  v'  )'  \ \ u \ \  dt.
\end{equation}

\noindent
Also note that:

$$( (1-t)^{\alpha +1} (1+t)^{\beta +1} u'  )'  = (1-t)^{\alpha }(1+t)^{\beta} \left[  (1-t^2 ) u'' -(\alpha +1 ) (1+t) u' + (\beta +1 ) (1-t) u'  \right] $$

$$= 
 (1-t)^{\alpha }(1+t)^{\beta} (  (1-t^2 ) u'' +(\beta -\alpha 
-(\alpha +\beta +2) t  ) u' ).$$

\noindent
In particular for $u= P_k^{(\alpha , \beta )} $ we get 

$$( (1-t)^{\alpha +1} (1+t)^{\beta +1}( P_k^{(\alpha , \beta )} ) '  )'  = -k(k+\alpha +\beta +1)  (1-t)^{\alpha} 
(1+t)^{\beta} P_k^{(\alpha , \beta )} ,$$ 

\noindent
and
setting $u= P_k^{(\alpha , \beta )}$, $v=P_j^{(\alpha , \beta )}$ in (\ref{Parts}), with $j \neq k$, we obtain

$$k(k+\alpha + \beta +1) \int_{-1}^{1} P_k^{(\alpha , \beta )}P_j^{(\alpha , \beta )} (1-t)^{\alpha} (1+t)^{\beta} dt =j (j+\alpha + \beta +1 ) 
 \int_{-1}^{1} P_j^{(\alpha , \beta )}P_k^{(\alpha , \beta )} (1-t)^{\alpha} (1+t)^{\beta} dt =0.$$

Since $  P_k^{(\alpha , \beta )} $ has degree $k$ it follows that  $ ( P_k^{(\alpha , \beta )}  )_{k\geq 0}$ forms a complete orthogonal system for the weighted $L^2$-space and therefore 
$k(k+\alpha + \beta +1)$, $k\geq 0$ are 
all the  eigenvalues for the eigenvalue problem (\ref{JacobiEquationL}).

\vspace{1cm}

Gegenbauer polynomials are particular cases of Jacobi polynomials, obtained when $\alpha = \beta$. For $\alpha > -1/2$ we call $C_k^{\alpha} = P_k^{\alpha  -(1/2),\alpha -(1/2)}$.
The Gegenbauer polynomial $C_k^{\alpha} $ has degree $k$ and satisfies the Gegenbauer differential equation

\begin{equation}\label{GegenbauerEquation}
(1-t^2 )y'' (t) -(2 \alpha +1)t) y'(t) + k(k+2\alpha ) y(t) =0.
\end{equation}

The Legendre polynomials $P_k$ are particular cases of Gegenbauer polynomials, taking $\alpha = 1/2$: $P_k = C_k^{1/2} = P_k^{0,0}$. 

\vspace{1cm}

All Jacobi polynomials can be computed recursively as it was discussed above. One also has the explicit formula \cite[Lecture 2]{Askey}

$$P_k^{(\alpha , \beta )} (t)  =(1-t)^{-\alpha} (1+t)^{-\beta} \frac{(-1)^k }{2^k k!}  \frac{d^k}{dt^k} ( (1-t)^{k+\alpha} (1+t)^{k+\beta}  ) $$

We will use this formula to fix a normalization for $P_k^{(\alpha , \beta )} $ as in   \cite{Askey}. Note that with this normalization we have that
$P_0^{(\alpha , \beta )} (t) =1 $, $P_1^{(\alpha , \beta )} (t)= (1/2)(\alpha - \beta ) + (t/2) (\alpha +\beta +2)$. Also:

$$ P_k^{(\alpha , \beta )} (-1) = (-1)^k \frac{(\beta +1)_k}{k!}  =   (-1)^k \frac{(\beta +1)...(\beta +k)}{k!} ,$$

and

$$ P_k^{(\alpha , \beta )}  (1)=   \frac{(\alpha +1)_k}{k!}  =  \frac{(\alpha +1) ...(\alpha +k )}{k!} . $$

In particular $ P_k^{(\alpha , \beta )}  (1) >0$ while $ P_k^{(\alpha , \beta )} (-1) .   P_{k+1}^{(\alpha , \beta )} (-1)  <0  $. 
By a simple application of Sturm-Liouville comparison theorem, the number of zeroes in (-1,1) of a nontrivial solution of
Equation (\ref{JacobiEquationL}) is nondecreasing in $\lambda$.
Since
$P_1^{(\alpha , \beta )}$ has a zero in $\frac{\beta -\alpha}{\alpha + \beta +2} \in (-1,1)$, it follows by induction that $P_k^{(\alpha , \beta )}$ has
$k$ zeroes in $(-1,1)$.

\vspace{1cm}

The linearization problem for the family  $P_k^{(\alpha , \beta )}$, $k\geq 0$, of Jacobi polynomials consists in finding the
coefficients $C_{jk}^l$ such that

$$ P_k^{(\alpha , \beta )}P_j^{(\alpha , \beta )} = \sum_{l=0}^{j+k} C_{jk}^l P_l^{(\alpha , \beta )}$$

The linearization coefficients can be obtained for any particular case, but it is difficult to obtain general formulas. There are
expressions for the linearization coefficients using hypergeometric functions, but it is difficult to obtain general statements from these. 
We will summarize now a few things that are known. For more details 
see the discussion in \cite[Lecture 5]{Askey}.

The most classical case of Legendre polynomials is well understood. One of the methods used to obtain the linearization coefficients is to
obtain the fourth order differential equation satisfied by a product $P_k P_j$. Then plug in $\sum_{l=0}^{j+k} C_{jk}^l P_l^{(\alpha , \beta )}$
into  the equation and use known relationships between the polyomials to obtain the linearization coefficients. This is done in the classical book
by E. W. Hobson \cite[Capter 2]{Hobson}. The argument goes back to the work of F. E. Neumann \cite{Neumann}. 

The same argument applies for the linearization coefficients for the Gegenbauer polynomials. The formula for the linearization coefficients in this
case is known as Dougall´s formula. It was stated without a proof by J. Dougall in \cite{Dougall}. A proof of the formula was given  by H. Y. Hs\"{u} in \cite{Hsu}.  E. Hyllerass \cite{Hylleraas} gave
a proof of the formula using an argument similar to the one mentioned above for the case of the Legendre polynomials.
In Dougall's formula the linearizarion coefficients appear as simple products and so it is easy to check the sign of the coefficients. Such simple formula is
not known for general Jacobi polynomials. But Hylleraas obtained the differential equation satisfied by the product of general Jacobi polynomials and 
deduced relationships between the coefficients which help to understand them.

 Later G. Gasper \cite{Gasper} used these relationships to prove that if $\alpha \geq \beta$ and $\alpha + \beta +1 \geq 0$ then all the linearization
coefficients are nonnegative. It is easy to check using Gasper arguments when the coefficients are actually strictly positive. To make it simpler we
will only consider the case of the square ${P_k^{(\alpha , \beta )}}^2 $, since it is the case that we will need.

In the next theorem we will discuss how to prove this using Gasper's work, and we also state all the results mentioned before and that we will
need in the following section.

\begin{theorem}\label{Jacobi} Fix $\alpha , \beta >-1$ such that $\alpha \geq \beta$ and $\alpha  + \beta +1 > 0$. Consider the family of Jacobi polynomials 
$ P_k^{(\alpha , \beta )}$, $k\geq 0$. $ P_k^{(\alpha , \beta )}$ has $k$ zeros in $(-1,1)$. $ P_k^{(\alpha , \beta )} (1)>0$ while 
$ P_k^{(\alpha , \beta )} (-1) >0$  if $k$ is even and $ P_k^{(\alpha , \beta )} (-1) <0$ if $k$ is odd.

Let $ P_k^{(\alpha , \beta )}  P_k^{(\alpha , \beta )} = \sum_{i=0}^{2k} C_k^i  P_i^{(\alpha , \beta )}$. 

A. If $\alpha = \beta$ then: if $k$ is odd then $C_k^i =0$ if $i$ is odd. If $k$ is even then $C_k^i >0$ if $i$ is even. In particular since the
family is orthogonal with respect to the weight $(1-t)^{\alpha} (1+t)^{\beta}$ we have:

If $k$ is odd 

$$\int_{-1}^1  {P_k^{(\alpha , \beta )}}^3   (1-t)^{\alpha} (1+t)^{\beta} \ dt =0.$$ 

If $k$ is even 
$$\int_{-1}^1  {P_k^{(\alpha , \beta )}}^3   (1-t)^{\alpha} (1+t)^{\beta} \ dt >0.$$

B. If $\alpha > \beta$ then $C_k^i >0$ for all $i=0,...,2k$. In particular 

$$\int_{-1}^1  {P_k^{(\alpha , \beta )}}^3  (1-t)^{\alpha} (1+t)^{\beta} \  dt >0.$$

\end{theorem}

\begin{proof} The first statements have already been discussed  in this section. Also for the statement (A) we have pointed out that when $\alpha = \beta$, which is the
case of the Gegenbauer polynomials, $P_k^{(\alpha , \beta )}$ is odd if $k$ is ood and even if $k$ is even. Therefore if $k$ is odd, $[P_k^{(\alpha , \beta )} ]^2$ is
an even polynomial and therefore if is written as a linear combination of $P_{2i}^{(\alpha , \beta )}$, $i\geq 0$. This means $C_k^{i} =0$ if $i$ is odd. Then

$$\int_{-1}^1  {P_k^{(\alpha , \beta )}}^3   (1-t)^{\alpha} (1+t)^{\beta}   dt = \int_{-1}^1  {P_k^{(\alpha , \beta )}} \left(  \sum_{i=0}^{2k} C_k^i  P_i^{(\alpha , \beta )} \right)  
 (1-t)^{\alpha} (1+t)^{\beta}  dt  $$

$$= C_k^k \int_{-1}^1  {P_k^{(\alpha , \beta )}}^2   (1-t)^{\alpha} (1+t)^{\beta}  dt  =0.$$

For the case when $\alpha = \beta $ and $k$ is even and the case when $\alpha >\beta$ we will follow the argument in \cite{Gasper}. Let

$$R_k^{(\alpha , \beta )} = \frac{P_k^{(\alpha , \beta )} }{ P_k^{(\alpha , \beta )} (1)}.$$

\noindent
$R_k^{(\alpha , \beta )} $ is of course just a different normalization for the Jacobi polynomial; note that it is a positive 
multiple of $P_k^{(\alpha , \beta )} $ and therefore the sign of the linearization coefficients for $R_k^{(\alpha , \beta )} $ are the same
as the sign of the corresponding linearization coefficients for $P_k^{(\alpha , \beta )} $. To simplify the notation we will consider $k$ fixed and write

$$ R_k^{(\alpha , \beta )}  R_k^{(\alpha , \beta )} = \sum_{i=0}^{2k} G^i  R_i^{(\alpha , \beta )}.$$

It is pointed out in \cite[Page 173]{Gasper} that $G^i $ is a positive multiple of $d_i = (-1)^i c_i$, where $c_i$ are the linearization coefficients
of hypergeometric functions treated in \cite{Hylleraas}. Using the recurrence formula given by \cite[(4.13)]{Hylleraas}, G. Gasper obtained
the following recurrence relation for the coefficients $d_i$ \cite[formula (5)]{Gasper}: let $a=\alpha + \beta +1 >0$ and $b=\alpha - \beta \geq 0$. For $1\leq j \leq 2k-1$ we have

$$\frac{(j+1)(2j +1 + a+b)(2k+j+1+a)(2k-j-1+a)(j+1)}{(2j+1+a)(2j+2+a)} d_{j+1} $$

$$ =b \left( \frac{(j+1)^2  (2k+j+2a)(2k-j)}{2j+1+a}-\frac{j^2(2k+j-1+2a)(2k-j+1)}{2j-1+a} \right) \ d_j $$

$$+\frac{(j-1+a)(2j+2\beta )(2k+j-1+2a) (2k-j+1)(j-1+a)}{(2j-2+a)(2j-1+a)} \ d_{j-1} .$$

It is also proved by G. Gasper \cite[page 174]{Gasper} that $d_0 >0$ and $d_{2k} >0$. Let $A_{j}$, $B_{j}$, $C_j$, be the coefficients multiplying
$d_{j+1}, d_j$ and $d_{j-1}$ in the previous formula, respectively. Note that $A_j >0$ and $C_j >0$. 

In the case of Gegenbauer polynomials, $b=0$, and it follows that $d_i >0$ if $i$ is even, while $d_i =0$ if $i$ is odd. This completes the proof of statement (A).

Assume now that $b>0$.

For the cases $j=2k$ we have

$$b        \frac{4k^2 (4k-1+2a)}{4k-1+a}  \ d_{2k} = \frac{(2k-1+a)(4k+2\beta )(4k-1+2a) (2k-1+a)}{(4k-2+a)(4k-1+a)} \ d_{2k-1} .$$

And for the case $j=0$ we have

$$  \frac{(1 + a+b)(2k+1+a)(2k-1+a)}{(1+a)(2+a)} d_{1}  = b   \frac{(2k+2a)(2k)}{1+a} \ d_0 .$$

It follows directly from these that $d_1 >0$ and $d_{2k-1} >0$. In particular this already implies statement (B) for the case $k=1$. Assume then that $k>1$.

To study the sign of $B_j$ write

$$B_j = \frac{Q_j}{(2j+1+a)(2j-1+a)}.$$

We need to understand the sign of $Q_j$. It will be useful to let $J=j-1$. We have by a simple computation that $Q_J$ is a polynomial of degree 4 in $J$:

$$Q_J = (J+2)^2  (J+ 2k+2a +1)(2k-J-1)(2J+a+ 1)-(J+1)^2(J + 2k+ 2a)(2k-J)(2J+a+3)$$

$$=-6J^4 - 12(a+2)J^3  + (8k(k+a) -6a^2 -38a-34) J^2$$      

$$+ (8k(k+a)(a+2) -14a^2 -38a -20)J +4 (k+3ka +3a+1)(k-1) +4ak +4a^2 (3k-2 )  .     $$

\vspace{.5cm}

The coefficients multiplying $J^4$ and $J^3$ are negative. Since we assumed that $k\geq 2$, we have that 
the coefficients multiplying $J^1$ and $J^0$ are positive. Independently of the sign of the coefficient 
multiplying $J^2$ we have that the coefficients change sign exactly once, going from positive no negative either at
$J^2$ or at $J^3$. Then it follows from Descartes' rule of signs  that $Q_J$, as a function of a real variable, has
exactly one positive zero. since $Q(0)>0$ and $Q(+ \infty ) = -\infty$, it follows that there exists $x_0 >0$ sucht that 
$Q(x) >0$ if $x\in (0,x_0  )$ and $Q(x )<0$ if $x>x_0$. 

Note that if $B_j \geq  0$ then it follows that if $d_{j-1} , d_j >0$ then $d_{j+1} >0$. If $B_j <0$ then it follows that if $d_{j+1}, d_j >0$ then $d_{j-1} >0$.
Then if $B_j >0$ for all $1\leq j \leq 2k -1$ it follows that $d_j >0$ for all $0\leq j \leq 2k$. And also if 
$B_j <0$ for all $1\leq j \leq 2k -1$ then $d_j >0$ for all $0\leq j \leq 2k$. More generally, if $B_{j_0} \geq 0$ for some
$1\leq j_0 \leq 2k-1$ (which implies form the previous discussion that $B_j >0$ for all $j<j_0$) then $d_j >0$ for
all $1\leq j \leq j_0 +1$. And if  $B_{j_0 +1}<0$ for some
$1\leq j_0 \leq 2k-2$ (which implies from the previous discussion that $B_j <0$ for all $j \geq j_0 +1$) then $d_j >0$ for
all $j_0 \leq j \leq 2k-1$. The previous discussion says that there is a $j_0$ verifying these two conditions, and therefore
again $d_j >0$ for all $j$. This completes the proof of statement (B).

\end{proof}

\section{Local bifurcation diagrams}

We will consider an equation of the form

\begin{equation}\label{ODE}
(1-t^2 ) u'' (t) + (\beta - \alpha  -(\alpha +\beta +2 ) t ) u'(t) - \lambda (u -u^q ) (t) =0.
\end{equation}

\noindent
where $\alpha , \beta >-1$, $\lambda >0$, $q>1$. 
The equation is considered on the interval $[-1,1]$ and the solution
$u$ must satisfy the boundary conditions

\begin{equation}\label{C(-1)}
 (2\beta +2) u' (-1) - \lambda (u(-1) -u(-1)^q )=0.
\end{equation}

\begin{equation}\label{C(1)} 
 -(2\alpha +2) u' (1) - \lambda (u(1) -u(1)^q )=0.
\end{equation}

We consider the space $X =C^{2,\alpha} [-1,1] \times (0,\infty )  $. And $Y=C^{0,\alpha} [-1,1]$. Then we 
consider the map $F:X  \rightarrow Y$, $F(u,\lambda ) =  (1-t^2 ) u'' (t) +(\beta - \alpha -(\alpha + \beta +2)t) u'(t) - \lambda (u-u^q) (t)$.

We will
consider positive functions $u$, actually close to the constant solution $u=1$ (solution of $F(u,\lambda )=0$). 
Note that, for $v\in C^2 [-1,1]$, 

$$D_u F (u,\lambda )[v] = (1-t^2 ) v'' (t)+(\beta - \alpha -(\alpha + \beta +2)t)   v'(t)  - \lambda (1-qu^{q-1} )v(t) .$$

In particular

$$D_u F(1,\lambda) [v]= (1-t^2 ) v'' (t) +(\beta - \alpha -(\alpha + \beta +2)t)  v'(t)  - \lambda (1-q)v(t) . $$

\begin{lemma}\label{Adjunto} For any $v,w \in C^2 [-1,1]$
\begin{equation} 
\int_{-1}^{1} D_u F(1,\lambda) [v] \ w (1-t)^{\alpha} (1+t)^{\beta} \  dt = 
\int_{-1}^{1} D_u F(1,\lambda) [w] \ v (1-t)^{\alpha} (1+t)^{\beta} \  dt
\end{equation}
\end{lemma}

\begin{proof} By a direct computation

$$ ( (1-t)^{\alpha +1} (1+t)^{\beta +1}  v' )' = (1-t)^{\alpha} (1+t)^{\beta}  ((1-t^2 )v'' +(\beta -\alpha -(\alpha +\beta +2)t) \ v' ),$$ 

\noindent
and therefore

$$\int_{-1}^{1} D_u F(1,\lambda) [v] \ w  (1-t)^{\alpha} (1+t)^{\beta} \ dt = \int_{-1}^{1} (    (1-t)^{\alpha +1} (1+t)^{\beta +1}    v' )' w -  \lambda (1-q)v w 
 (1-t)^{\alpha} (1+t)^{\beta} \  dt $$

We apply integration by parts as in (\ref{Parts}) to the first term on the right to obtain:

$$\int_{-1}^{1} D_u F(1,\lambda) [v] \ w  (1-t)^{\alpha} (1+t)^{\beta} \ dt=\int_{-1}^{1} (     (1-t)^{\alpha +1} (1+t)^{\beta +1}        w' )' v -  \lambda (1-q)v w 
   (1-t)^{\alpha} (1+t)^{\beta} \        dt $$

$$=\int_{-1}^{1} D_u F(1,\lambda) [w] \ v      (1-t)^{\alpha} (1+t)^{\beta} \        dt .$$

\end{proof}

The following results gives a more explicit discussion for the local bifurcation for Equation (\ref{ODE}) and contains Theorem 1.2.

\begin{theorem}\label{LocalBifurcationODE} Let $\lambda_k =\frac{ k(k+\alpha +\beta +1)}{q-1}$.
If $\lambda \neq \lambda_k$ for all $k$ then the trivial family of solutions of Equation (\ref{ODE})
is locally rigid at $(1,\lambda )$. For
each $k\geq 1$, $(1,\lambda_k )$ is a bifurcation point for the family. The set of nontrivial solutions in a
neighborhood of $(1,\lambda_k )$ is given by a path $s\mapsto (u(s), \lambda (s))$, such that $u(0)=1$,
$\lambda (0)=\lambda_k$,  $u(s) \neq 1 $ if $s\neq 0$, and $\frac{du}{ds} (0) =P_k^{(\alpha , \beta )}$ (the Jacobi polynomial from the previous section).
We have:

$$\frac{d\lambda}{ds} (0)= C \int_{-1}^{1}  ( P_k^{(\alpha , \beta )} )^3  (1-t)^{\alpha} (1+t)^{\beta}   \ dt,$$

\noindent
where 

$$C= \frac{-q}{2  \int_{-1}^{1} ( P_i^{\alpha , \beta} )^2  (1-t)^{\alpha} (1+t)^{\beta} dt }$$

\noindent
is a negative constant. Therefore if $\alpha > \beta$ or $\alpha = \beta$ and $k$ is even we have that $\frac{d\lambda}{ds} (0) <0$.
If $\alpha = \beta$ and $k$ is odd then $\frac{d\lambda}{ds} (0) =0$.

\end{theorem}

\begin{proof} We apply the classical resut on bifurcation from simple eigenvalues (see \cite{Crandall, Nirenberg}) to the operator $F:X  \rightarrow Y$.
$ND_u F (1,\lambda_k ) = \langle P_k^{\alpha , \beta} \rangle$ and it follows from Lemma \ref{Adjunto} that $D_u F (1,\lambda_k )$ is self adjoint with respect to
the $(1-t)^{\alpha} (1+t)^{\beta}$-weighted $L^2$ product on $[-1,1]$. The range of $D_u F (1,\lambda_k )$, $R$, is the space orthogonal to $P_k^{\alpha , \beta}$
and $D_{u \lambda} F (1,\lambda_k ) [ P_k^{\alpha , \beta} ] =-(1-q) P_i^{\alpha , \beta} \notin R$. Then the Bifurcation from Simple Eigenvalues Theorem 
(\cite[Theorem 1]{Crandall}, \cite[Theorem 3.2.2]{Nirenberg}) says that in a neighborhood of $(1,\lambda_k )$ the space of solutions of $F(u,\lambda )=0$ is
given by the path of trivial solutions $(1,\lambda )$ and a path of nontrivial solutions $s\mapsto (u(s),\lambda (s) )$ verifying: $\lambda (0)=\lambda_k$,
$u(0)=1$, $\frac{\partial u}{\partial s} (0)= P_k^{\alpha , \beta}$. 

We will use $ '$ to denote differentiation with respect to the variable $t\in [-1,1]$. We have

$$(1-t^2) u(s)'' + (\beta - \alpha -(\alpha + \beta +2)t) u(s)' - \lambda(s) (u(s)-u(s)^q) =0. $$

Then differentiating with respect to $s$ we get:

$$(1-t^2) \left( \frac{du}{ds} \right)  '' + (\beta - \alpha -(\alpha + \beta +2)t) \frac{du}{ds} ' - \lambda (s)  \left( \frac{du}{ds}-qu(s)^{q-1} \frac{du}{ds} \right) - \frac{d \lambda}{ds} (u(s)-u(s)^q)  =0. $$

And differentiating with respect to $s$ once more we get:

$$(1-t^2)  \left(   \frac{d^2 u}{ds^2} \right) '' + (\beta - \alpha -(\alpha + \beta +2)t) \left(  \frac{d^2 u}{ds^2} \right)  ' - \lambda (s)  \left( \frac{d^2 u}{ds^2}-qu(s)^{q-1} \frac{d^2 u}{ds^2} -q(q-1)u(s)^{q-2}  \left(   {\frac{du}{ds}}  \right)^2  \right) $$

$$-2\frac{d \lambda}{ds} \left( \frac{du}{ds}-qu(s)^{q-1} \frac{du}{ds}  \right)
-\frac{d^2 \lambda}{ds^2} (u(s)-u(s)^q) =0. $$

Evaluating at $s=0$ we obtain:

$$(1-t^2 ) \left(  \frac{d^2 u}{ds^2} (0) \right)  '' + (\beta - \alpha -(\alpha + \beta +2)t) \left(  \frac{d^2 u}{ds^2} (0) \right)  '  $$

$$-\lambda_k \left( (1-q) \frac{d^2 u}{ds^2} (0) -q(q-1) ( P_k^{\alpha , \beta} )^2 \right)
-2 \frac{d \lambda}{ds} (0) (1-q)  P_k^{\alpha , \beta} =0.$$

We let $w= \frac{d^2 u}{ds^2} (0)$, multiply by $P_k^{\alpha , \beta}    (1-t)^{\alpha} (1+t)^{\beta}    $ and integrate:

$$\int_{-1}^{1} D_u F (1,\lambda_k ) [w] P_k^{\alpha , \beta} (1-t)^{\alpha} (1+t)^{\beta} dt +\lambda_k q (q-1)  \int_{-1}^{1} ( P_k^{\alpha , \beta} )^3  (1-t)^{\alpha} (1+t)^{\beta} dt $$

$$+2 (q-1)\frac{d \lambda}{ds} (0) \int_{-1}^{1} ( P_k^{\alpha , \beta} )^2  (1-t)^{\alpha} (1+t)^{\beta} dt  =0 $$

Using Lemma \ref{Adjunto}, since  $ D_u F (1,\lambda_k ) [P_k^{\alpha , \beta} ]=0$ we see that the first term vanishes and then we obtain:

$$\frac{d \lambda}{ds} (0)  = \frac{  -q \lambda_k  \int_{-1}^{1} ( P_k^{\alpha , \beta} )^3  (1-t)^{\alpha} (1+t)^{\beta} dt }{2  \int_{-1}^{1} ( P_k^{\alpha , \beta} )^2  (1-t)^{\alpha} (1+t)^{\beta} dt } .$$

The last statement of the theorem then follows from Theorem \ref{Jacobi}.

\end{proof}

\section{Degenerate solutions}  

In this section we will prove Theorem 1.3. Fix the isoparametric function $f^{d,c}$. We will assume that $c \leq 0$, as it was mentioned in 
the introduction.
Also we pick $\alpha , \beta$ by solving 

$$\beta - \alpha = (1/2)c $$

$$\alpha + \beta +2 = \frac{n+d-1}{d},$$

 A $f^{d,c}$-invariant solution of Equation (\ref{equation}) is given by a function $\varphi \in C^2 [-1,1]$ which solves

\begin{equation}\label{YamabeJacobi}
 (1-t^2 ) \varphi '' (t) + (-(\alpha +\beta +2)t + \beta -\alpha ) \varphi '(t) - \frac{ \lambda }{d^2}  ( \varphi (t) -\varphi^q (t)) =0,
\end{equation}

\noindent
where we point out that we have that $\alpha \geq \beta$ and  $\alpha + \beta +1 > 0$.

We recall the following two results which appear in \cite[Theorem 1.2, Theorem 4.1]{AJJ}:

\begin{theorem}\label{CloseToZero} If $1<q<q_f$ there exists $\lambda_0 >0$ such that if $u$ is a positive $f^{d,c}$-invariant
solution 
of Equation (\ref{equation}) with $\lambda <\lambda_0$ then $u=1$.
\end{theorem}

\begin{theorem}\label{Compact} If $1<q<q_f$. For any positive numbers $0<\varepsilon < M$ we have that the space of positive $f^{d,c}$-invariant
solutions 
of Equation (\ref{equation}) with $\lambda \in [ \varepsilon , M ]$ is compact. 

\end{theorem}

Now we can prove Theorem 1.3.

\begin{proof} A  $f^{d,c}$-invariant solution of Equation (\ref{equation}) is given by a solution $u$ of
Equation (\ref{YamabeJacobi}).
Consider the space of nontrivial positive  $f^{d,c}$-invariant solutions of Equation (\ref{equation}):
 $D =\{ (u , \lambda) \in (C^{2,\alpha} [-1,1]  - \{ 1 \} ) \times (0, \infty ) : u \;\text{is a positive 
nontrivial solution of}\; (\ref{YamabeJacobi}) \}$. Consider the bifurcation points  $(1, d^2 \lambda_k)$ given by
Theorem \ref{LocalBifurcationODE}. Consider also the path of nontrivial
solutions $s\mapsto (u(s), \lambda (s) )$, $s\in (-\varepsilon , \varepsilon )$.
Recall that we have that $\alpha \geq \beta$ and  $\alpha + \beta +1 > 0$, as needed for Theorem \ref{LocalBifurcationODE}. 
Since $u(0)=1$, $\frac{d u}{ds} (0) = P_k^{\alpha , \beta}$ and $ P_k^{\alpha , \beta}$ has $k$ zeroes in $(-1,1)$.
It follows that for $s$ close to zero, $s\neq 0$, the solution $u(s)$ takes the value 1 exactly $k$-times.

 Let $D_k^+$ be the connected 
component of $D$ containing the path $(u(s), \lambda (s))$ with $s>0$ and 
$D_k^-$ be the connected 
component of $D$ containing the path $(u(s), \lambda (s))$ with $s<0$. Let $D_k =D_k^+ \cup D_k^-$. Note that since 
$ P_k^{\alpha , \beta} (1) >0$ we have that for $s>0$, $s$ small, $u(s)(1) >1$. On the other hand if at some point $s>0$
we have that $u(s)(1) =1$ then we would have from Equation (\ref{ODE}) that $u$ is constant (equal to 1).

Let $(u,\lambda) \in D$. Then at any point $t_0$ such that $u(t_0 )=1$ we have $u' (t_0 ) \neq 0$. Note also
that $u(-1)\neq 1$ since $u$ must satisfy \ref{C(-1)}. If
$u(-1 ) >1$ ($<1$) then there exists a neighborhood $U$ of $(u,\lambda )$ in $C^{2,\alpha }[-1,1] \times \re$ such that
for any $(v,\mu ) \in U$ we have $v(-1 ) >1$ ($<1$)  and $v$ and $u$ take the value 1 the same number of times.
It follows that on a connected subset $C$ of $D$ we have for all $(u,\lambda )\in C$ the number of times $u$ takes
the value $1$ is the same, and also the sign of $u(-1)-1$. Therefore for all $(u,\lambda )\in D_k$ we have that
$u^{-1} (1)$ has $k$ elements. 
It follows then from this discussion  that if $(1,\lambda) \in \overline{D_k}$
then $\lambda = d^2 \lambda_k $. 

Theorem \ref{LocalBifurcationODE} says that  $\frac{d \lambda}{ds} (0) <0$. Then for $s>0$ small, we have that $u(s)$ is
a solution of Equation (\ref{YamabeJacobi}) with $\lambda = \lambda (s) < d^2  \lambda_k$.

Now it follows from Theorem \ref{CloseToZero} that if $( u,\lambda  ) \in D_k^+$ then $\lambda \geq \lambda_0$. From
Theorem \ref{Compact} we know that $\{ (u,\lambda ) \in D_k : \lambda \in [\lambda_0 , d^2 \lambda_k ] \}$ is compact. 
Then there exists $(u_* , \lambda_* ) \in \overline{D_k^+}$ such that $\lambda_* \leq  \lambda$ for all 
$(u,\lambda )\in D_k^+$. But if $u_*$ is constant then  $(1, \lambda_* )$ would be a bifurcation point and
therefore $\lambda_* = d^2 \lambda_i$ for some $i<k$. But we already mentioned that this cannot happen. 
Therefore $u_*$ is not constant and therefore $(u_* , \lambda_* ) \in D_k^+$. 

If $D_u F  (u_* , \lambda_* )$ were an isomorphism then in particular $DF$ would be surjective and we could
apply the implicit function theorem. Note that $D_{\lambda} F (u_* , \lambda_* ) = u_* - u_*^q \neq 0$ and
the kernel of $DF$ is one-dimensional. Therefore there exists a regular path $s \mapsto (u(s), \lambda (s))$ sucht that
$F(u(s), \lambda (s))=0$, $u(0)=u_*$, $\lambda (0)= \lambda_*$, $(u(s), \lambda (s) \in D_k^+$. Since
$\lambda_*$ is the minimum value of $\lambda$ for $(u,\lambda ) \in D_k^*$ we must have that $\lambda '(0)=0$. 
Then differentiating with respect to $s$ the equation

$$0=F(u(s), \lambda (s)) = -\Delta u(s) + \lambda (s) (u(s) - u(s)^q ),$$

\noindent
and evaluating at $s=0$, we obtain for $v=\frac{d u}{ds} (0) \neq 0$,

$$-\Delta v + \lambda_* (v-qu_*^{q-1} v ) =0.$$

This means that $v\in Kernel (D_u F  (u_* , \lambda_* )$ which is a contradiction. Therefore it is not true that 
$D_u F  (u_* , \lambda_* )$ is an isomorphism, which means that
$u_*$ is a degenerate, $f^{d,c}$-invariant solution of Equation (\ref{equation}). Note that if $(u,\lambda ) \in D$
and $t_0 \in (-1,1)$ is a critical point of $u$ we have that if $u(t_0 )<1$ the $t_0$ is a local minimum and if 
 $u(t_0 )<1$ the $t_0$ is a local maximum. It follows that the number of critical poins of $u\in D_k^+$ is 
exactly $k-1$. Therefore using Remark 1.4, we have proved the theorem.

\end{proof}

\end{document}